\documentclass[11pt,oneside,article]{memoir}

\settrims{0pt}{0pt} 
\settypeblocksize{*}{34.4pc}{*} 
\setlrmargins{*}{*}{1} 
\setulmarginsandblock{1.0in}{1.3in}{*} 
\setheadfoot{\onelineskip}{2\onelineskip} 
\setheaderspaces{*}{1.5\onelineskip}{*} 
\checkandfixthelayout

\usepackage{amsmath}
\usepackage{amssymb}
\usepackage{stix}
\usepackage{amsfonts}
\usepackage{amsthm}
\usepackage{mathtools}
\usepackage{newpxtext}
\usepackage{enumitem}
\usepackage[colorlinks,linkcolor=darkblue,citecolor=darkblue,urlcolor=darkblue,breaklinks=true]{hyperref}
\usepackage{tikz-cd}
\usepackage[capitalize]{cleveref}
\usepackage{todonotes}
\usepackage{graphicx}

\usepackage{eucal}

\usetikzlibrary{
	arrows.meta,
	}

 \tikzcdset{arrow style=tikz, diagrams={>=To}}

\newcommand{\drpullback}{\arrow[phantom]{dr}[very near start,description]{\lrcorner}}

  \setlist{nosep}

\newtheorem*{theorem*}{Theorem}
\newtheorem{definition}{Definition}[chapter]
\newtheorem{proposition}[definition]{Proposition}   
\newtheorem{theorem}[definition]{Theorem}
\newtheorem{lemma}[definition]{Lemma}   
\newtheorem{corollary}[definition]{Corollary}   

\theoremstyle{remark}

\newtheorem*{example*}{Example}
\newtheorem{remark}[definition]{Remark}

\crefalias{chapter}{section}

\let\OLDthebibliography\thebibliography
\renewcommand\thebibliography[1]{
  \OLDthebibliography{#1}
  \setlength{\parskip}{4pt}
  \setlength{\itemsep}{1pt plus 0.3ex}
}

\definecolor{darkblue}{rgb}{0,0,0.7}

\DeclareMathOperator{\id}{id}

\DeclareMathOperator{\cod}{cod}

\DeclareMathOperator{\ob}{Ob}

\newcommand{\last}{\Fun{last}}

\newcommand{\cat}[1]{\mathcal{#1}}
\newcommand{\Cat}[1]{\mathsf{#1}}
\newcommand{\Fun}[1]{\operatorname{#1}}

\newcommand{\uint}{\rotatebox{15}{\!$\smallint$\!}}
\newcommand{\ourint}{\uint}

\newcommand{\smset}{\Cat{Set}}
\newcommand{\smcat}{\Cat{Cat}}
\newcommand{\decomp}{\Cat{Decomp}}
\newcommand{\sset}{\Cat{sSet}}

\newcommand{\cocolon}{:\!}

\newcommand{\To}[1]{\xrightarrow{#1}}

\renewcommand{\ss}{\subseteq}
\newcommand{\op}{^\mathrm{op}}

\newcommand{\tw}{\Fun{tw}}
\newcommand{\QQQ}{Q}
\newcommand{\sd}{\Fun{sd}}

\newcommand{\untw}[1]{\Fun{untw}_{#1}}
\newcommand{\Psh}{\Cat{Psh}}
\newcommand{\RFib}{\Cat{RFib}}
\newcommand{\el}{\Fun{el}}

\newcommand{\ds}{\cat}
\newcommand{\dsD}{\ds{D}}

\DeclareMathAlphabet{\mathbbe}{U}{bbold}{m}{n}
\newcommand{\sxcat}{\mathbbe{\Delta}}

\newcommand{\slice}[1]{_{/#1}}

\newcommand{\activeto}{\rightarrowbar}

\newcommand{\LMO}[2][over]{\ifthenelse{\equal{#1}{over}}{#2}{#2}}

\newcommand{\nerve}{N}

\newcommand{\isopil}{\stackrel{\raisebox{0.1ex}[0ex][0ex]{\(\sim\)}}%
			{\raisebox{-0.15ex}[0.28ex]{\(\rightarrow\)}}}

\newcommand{\rat}{\rightarrowtail}

\newcommand{\join}{\star}

\newcommand{\Addresses}{{
  \bigskip
  \footnotesize

  J. Kock, \textsc{Departament de Matem\`{a}tiques, 
Universitat Aut\`{o}noma de Barcelona
}\par\nopagebreak
  \textit{E-mail address}: \texttt{kock@mat.uab.cat}

  \medskip

  D.I.~Spivak, \textsc{Department of Mathematics, Massachusetts Institute of Technology
}\par\nopagebreak
  \textit{E-mail address}: \texttt{dspivak@gmail.com}

}}

\begin{document}

\title{Decomposition-space slices are toposes}

\author{
Joachim Kock
\thanks{
Kock was supported by grants
MTM2016-80439-P  (AEI/FEDER, UE) of Spain and
  2017-SGR-1725 of Catalonia.}

\and
 David I. Spivak
\thanks{
Spivak was supported by AFOSR grants 
FA9550--14--1--0031 and FA9550--17--1--0058.
}
}

\date{\vspace{-1.5cm}}

\maketitle

\begin{abstract}
  We show that the category of decomposition spaces and CULF maps is
  locally a topos.  Precisely, the slice category over any
  decomposition space $\dsD$ is a presheaf topos, namely
  $\decomp\slice{\dsD}\simeq\Psh(\tw \dsD)$.
\end{abstract}

\chapter{Introduction}


  Decomposition spaces were introduced for purposes in
  combinatorics by G\'alvez, Kock, and Tonks~\cite{GKT1,GKT2,GKT3}, 
  and purposes in homological algebra, representation theory, and geometry
  by Dyckerhoff and Kapranov~\cite{Dyckerhoff-Kapranov:1212.3563},
  who call them unital $2$-Segal spaces.
They are simplicial sets (or
simplicial $\infty$-groupoids) with a property that expresses the
ability to (co-associatively) decompose, just as in categories
one can (associatively) compose.  In particular, decomposition spaces
induce coalgebras.
The most nicely-behaved class of
morphisms between decomposition spaces is that of CULF maps.
These preserve decompositions in an appropriate
way so as to induce coalgebra homomorphisms.

Apart from 
the coalgebraic aspect,
not so much is known about the category $\decomp$ of decomposition 
spaces and CULF maps, and it may appear a bit peculiar.
For example, the product of two decomposition 
spaces as simplicial sets is not the categorical product in 
$\decomp$---the projections generally fail to be CULF.  Similarly, the terminal 
simplicial set (which is a decomposition space) is not terminal in 
$\decomp$. 
The simplicial-set product should rather be considered
as a tensor product for decomposition spaces.

The present contribution advances the categorical study of 
decomposition spaces by establishing that $\decomp$ is locally a topos, meaning
that all its slices are toposes---even presheaf toposes.  More precisely we show:

\medskip

\noindent
\textbf{Main Theorem} (Theorem~\ref{thm.main}).
{\em 
	For $\dsD$ a decomposition space, there is a natural equivalence of 
	categories
	\begin{equation}\label{eqn.decomp_psh}
	\decomp\slice{\dsD} \isopil \Psh(\tw (\dsD)) .
	\end{equation}
}


\nopagebreak
\noindent 
Here $\tw(\dsD)$ is the twisted arrow category of $\dsD$, obtained
by
edgewise subdivision---this is
readily seen to be (the nerve of) a category, cf.~\Cref{lemma.tw_rFib} below.

\pagebreak

The functor in the direction displayed in \cref{eqn.decomp_psh} is 
simply applying the twisted arrow category construction.  
The functor going in 
the other direction is the surprise.  We construct it by exploiting
an interesting
natural transformation between the category of elements of the nerve
and the twisted arrow category of a decomposition space:
\[
\begin{tikzcd}[column sep=large]
	\decomp\arrow[r, bend left=20pt, start anchor={[xshift=-4pt, 
	yshift=-2pt]}, "\el\circ \nerve"]\arrow[r, bend right=20pt, start anchor={[xshift=-4pt, yshift=2pt]}, "\tw"']
	\ar[r, phantom, "\scriptstyle\Downarrow\lambda"]&
	\smcat  .
\end{tikzcd}
\]
Its component
$\lambda_\dsD\colon \el(\nerve \dsD) \to \tw(\dsD)$ 
at a decomposition space $\dsD$ 
sends an $n$-simplex of $\dsD$ to
its long edge.
(In the category case, this map goes back to
   Thomason's notebooks \cite[p.152]{Thomason:notebook85}; it was exploited
   by G\'alvez--Neumann--Tonks~\cite{Galvez-Neumann-Tonks:2013} to exhibit
   Baues--Wirsching cohomology as a special case of
   Gabriel--Zisman cohomology.)
The crucial property of $\lambda$ is that it is a \emph{cartesian} natural 
transformation.
Our proof of the Main Theorem
describes the inverse to $\tw_\dsD$
as being essentially---modulo some technical 
translations involving elements, presheaves, and nerves---given by $\lambda_\dsD^*$.

The importance of this result resides in making a huge body of work in
topos theory available to study decomposition spaces, such as for
example the ability to define new decomposition spaces from old by
using the internal language.  
It also
opens up interesting questions such
as how the subobject classifier \cite{MacLane-Moerdijk}, isotropy group
\cite{FHS}, etc.\ of the topos $\decomp\slice{\dsD}$ relate to the
combinatorial structure of the decomposition space $\dsD$.

The Main Theorem
can be considered surprising in view of
the well-known failure of such a result for categories. Lamarche
(1996) had suggested that categories CULF over a fixed base
category $\cat{C}$ form a topos, but
Johnstone~\cite{Johnstone:Conduche'} found a counter-example,
and corrected the statement by identifying the precise---though
quite restrictive---conditions that $\cat{C}$ must satisfy.
Independent proofs of this result (and related conditions)
were provided by Bunge--Niefield~\cite{Bunge-Niefield} and
Bunge--Fiore~\cite{Bunge-Fiore}, who were motivated by CULF
functors as a notion of duration of processes, as already
considered by Lawvere~\cite{Lawvere:statecats}.
The present work also grew out of interest
in dynamical systems \cite{Schultz-Spivak-Vasilakopoulou}.
Our main theorem can be seen as a different
realization of
Lamarche's insight, allowing more general domains, 
thus indicating a role of decomposition spaces
in category theory.  
From this perspective, the point is that the
natural setting for CULF functors are decomposition spaces
rather than categories: a simplicial set CULF over a category
is a decomposition space, but not always a category. 

%
%

\chapter{Preliminaries}

\paragraph{Simplicial sets.}
  Although decomposition spaces naturally pertain to the realm of
  $\infty$-categories and simplicial $\infty$-groupoids, we work in
  the present note with $1$-categories and simplicial {\em sets}, both
  for simplicity and in order to situate decomposition spaces
  (actually just ``decomposition sets'') in the setting of classical
  category theory and topos theory.  All the results should generalize
  to $\infty$-categories (in the form of Segal spaces) and general
  decomposition spaces, and the proof ideas should also scale to this
  context, although the precise form of the proofs does not: where
  presently we exploit objects-and-arrows arguments, more uniform
  simplicial arguments are required in the $\infty$-case.  We leave
  that generalization open.

Thus our setting is the category $\sset=\Psh(\sxcat)$ of simplicial
sets, i.e.~functors $ \sxcat\op\to\smset$ and their natural
transformations. Small categories
fully faithfully embed as simplicial sets via the nerve functor
$\nerve\colon\smcat\to\sset$, and decomposition spaces are defined as
certain more general simplicial sets, as we now recall.

\paragraph{Active and inert maps.}
  The category $\sxcat$ has an active-inert factorization system:
  the {\em active maps}, written $g\colon[m]\activeto [n]$, are those that  
  preserve end-points,
  $g(0)=0$ and $g(m)=n$; the {\em inert maps}, written $f\colon [m]\rat [n]$, are
  those that are distance preserving,
  $f(i{+}1)=f(i)+1$ for $0\leq i\leq m-1$.  
  The active maps are generated by
  the codegeneracy maps $s^i\colon X_n \to X_{n+1}$ (for 
  all $0 \leq i \leq n$) and the inner coface maps $d^i\colon X_n 
  \to X_{n-1}$ (for 
  all $0 < i < n$); the inert maps are
  generated by the outer coface maps $d^\bot=d^0$ and 
  $d^\top=d^n$.
  (This orthogonal factorization system is an instance of the 
  important general notion of generic-free factorization system of
  Weber~\cite{Weber:TAC18} who referred to the two classes as generic
  and free. The active-inert terminology is due to
  Lurie~\cite{Lurie:HA}.)

\paragraph{Decomposition spaces.}
  Active and inert maps in $\sxcat$ admit pushouts along each other,
  and the resulting maps are again active and inert.  A {\em
  decomposition space} \cite{GKT1} is a simplicial set (or
  more generally a simplicial groupoid or $\infty$-groupoid) $X\colon \sxcat\op\to\smset$ that
  takes all such active-inert pushouts to pullbacks:
  \[
  X\left(
  \begin{tikzcd}
  	{[n']} \drpullback &  {[n]} \ar[l, >->]\\
		{[m']} \ar[u, ->|] & 
		{[m]} \ar[u, ->|] \ar[l, >->]
  \end{tikzcd}
  \right)
  \qquad=\quad
  \begin{tikzcd}
  	X_{n'}\ar[r]\ar[d]&
		X_n\ar[d]\\
		X_{m'}\ar[r]&
		X_{m}\ar[ul, very near end, phantom, "\lrcorner"]   .
  \end{tikzcd}
  \]

\paragraph{CULF maps.}
  A simplicial map $F\colon Y\to X$ between simplicial sets is 
  called \emph{CULF}~\cite{GKT1} when it is cartesian on active maps 
  (i.e.\ the naturality squares are pullbacks), or, equivalently,
  is right-orthogonal to all active maps $\Delta[m]\activeto\Delta[n]$.
  The CULF maps between (nerves of) categories are precisely the
  discrete Conduch\'e fibrations (see~\cite{Johnstone:Conduche'}).
  If $\dsD$ is a decomposition space (e.g.~a category) and 
  $F\colon \ds{E} \to \dsD$ is CULF,
  then also $\ds{E}$ is a decomposition space (but not in general
  a category).\label{page.CULF_over_decomp}
  
  We denote by $\decomp$ the category of decomposition spaces and 
  CULF maps.

\paragraph{Right fibrations and presheaves.}

  A simplicial map is called a {\em right fibration} if it is
  cartesian on bottom coface maps (or equivalently, right-orthogonal
  to the class of last-vertex inclusion maps $\Delta[0]\to\Delta[n]$).
  (This in fact implies that it is cartesian on all codegeneracy and
  coface maps except the top coface maps. In particular, a right
  fibration is CULF.) When restricted to categories, this notion
  coincides with that of discrete fibration. We denote by $\RFib$ the
  category of small categories and right fibrations. Note that for any
  category $\cat{C}$, the inclusion functor $\RFib\slice{\cat{C}} \to
  \smcat\slice{\cat{C}}$ is full.
  
  For any small category $\cat{C}$ there is an adjunction 
  \[
  \partial\colon\smcat\slice{\cat{C}}\leftrightarrows\Psh(\cat{C})\cocolon\ourint
  \] 
  with $\ourint\vdash\partial$; the presheaf $\partial(F)$
  associated to a functor $F\colon\cat{D}\to\cat{C}$ is given by
  taking the left Kan extension $\partial(F)\coloneq \Fun{Lan}_F(1)$
  of the terminal presheaf $1\in\Psh(\cat{D})$ along $F$.
  The notation is chosen
  because $\partial\circ\ourint\cong\id$. For any $X\in\Psh(\cat{C})$,
  we denote $\ourint X$ by $\pi_X\colon\el(X)\to\cat{C}$, and refer to
  $\el(X)$ as the \emph{category of elements}; an object of
  $\el(X)$ is a pair $(c,x)$ where $c\in\cat{C}$ and $x\in X(c)$, and
  a morphism $(c',x')\to (c,x)$ is a morphism $f\colon c'\to c$ in
  $\cat{C}$ with $X(f)(x)=x'$. The projection functor $\pi_X$ is
  always a right fibration, and $\partial$ restricts to an equivalence
  of categories $\RFib\slice{\cat{C}}\simeq\Psh(\cat{C})$. If $X'\to
  X$ is a map of presheaves, $\el(X')\to\el(X)$ is a right fibration.
  Thus we have a functor $\el\colon\Psh(\cat{C})\to\RFib$. We will
  make particular use of this for the case $\cat{C}=\sxcat$:
  \[
\el\colon\sset\to\RFib.
\]
  
\paragraph{Twisted arrow categories.}
  For $\cat{C}$ a small category, the {\em twisted arrow category}
  $\tw(\cat{C})$ (cf.~\cite{Lawvere:hyperdoc}) is the category of 
  elements of the Hom functor $\cat{C}\op\times \cat{C} \to \smset$.
  It thus 
  has the arrows of $\cat{C}$ as objects, and trapezoidal commutative diagrams
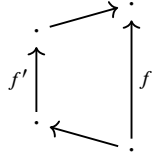
\[
\begin{tikzcd}[row sep=0, cramped]
	&\cdot\\
	\cdot \ar[ur] \\[25pt] 
	\cdot \ar[u, "f'"]\\
	&\cdot \ar[ul] \ar[uuu, "f"']
\end{tikzcd}
\]
as morphisms from $f'$ to $f$.

The twisted arrow category is a special case of 
\emph{edgewise subdivision} of a simplicial set~\cite{Segal:1973}, 
as we now recall.  Consider the functor 
\begin{eqnarray*}
	\QQQ\colon  \sxcat & \longrightarrow & \sxcat \\
	{}[n] & \longmapsto & [n]\op\join [n] = [2n{+}1]  .
\end{eqnarray*}
With the following special notation for the elements 
of the ordinal
$[n]\op\join [n] = [2n{+}1]$,
	\begin{equation}\label{eqn.twisted_n}
	\begin{tikzcd}[cramped, row sep=small, column sep=small]
		\LMO{0}\ar[r]&\LMO{1}\ar[r]&\cdots\ar[r]&\LMO{n}\\
		\LMO[under]{0'}\ar[u]&\LMO[under]{1'}\ar[l]&\cdots\ar[l]&\LMO[under]{n'} ,
		\ar[l]
	\end{tikzcd}
	\end{equation}
the functor $\QQQ$ is 
described on arrows by sending a
coface map $d^i \colon [n{-}1]\to [n]$ to the monotone map that omits the 
elements $i$ and $i'$, and by sending a codegeneracy map $s^i \colon [n] 
\to [n{-}1]$ to the monotone map that repeats both $i$ and $i'$. 

Defining $\sd\coloneq \QQQ^*\colon\sset\to\sset$, we have the commutative diagram
	\[
	\begin{tikzcd}
		\smcat\ar[r, "\tw"]\ar[d, "\nerve"']&
		\smcat\ar[d, "\nerve"]\\
		\sset\ar[r, "\sd"']&
		\sset  .
	\end{tikzcd}
	\]

Note that there is a natural transformation 
$L\colon\id_\sxcat\Rightarrow\QQQ$, 
whose component at $[n]$ is the last-segment inclusion $[n]\ss[n]\op\join[n]$. 
It induces a natural map
\[\cod_X \coloneq L^*\colon\sd(X)\to X\]
for any simplicial set $X$,
and similarly $\tw(\cat{C})\to \cat{C}$ for any category $\cat{C}$.

\begin{lemma}\label{lemma.tw_rFib}
	The functor $\sd\colon\sset\to\sset$ sends decomposition spaces 
	to categories and CULF maps to right fibrations. That is, there 
	is a unique
	functor $\tw$ making the following diagram commute:%
	\[
	\begin{tikzcd}
		\decomp\ar[r, dotted, "\tw"]\ar[d, "\nerve"']&
		\RFib\ar[d, "\nerve"]\\
		\sset\ar[r, "\sd"']&
		\sset   .
	\end{tikzcd}
	\]
\end{lemma}
Note that the ``nerve functor'' 
$\nerve \colon \decomp 
\to \sset$ on the left of the diagram is just the inclusion, 
but it will be convenient to have it named, so as to stress
that we regard decomposition spaces as structures generalizing 
categories.
\begin{proof}
	Suppose $\dsD$ is a decomposition space; we need to check that in
	$\sd(\nerve\dsD)$, the inert face maps $d_\top$ and $d_\bot$ form
	pullbacks against each other (the Segal condition). But each top
	face map in $\sd(\nerve\dsD)$ is given by the composite of two
	outer face maps in $\nerve\dsD$ (removing $n$ and $n'$ in
	\cref{eqn.twisted_n}), and each bottom face map in
	$\sd(\nerve\dsD)$ is given by the composite of two inner face maps
	in $\nerve\dsD$ (removing $0$ and $0'$). Hence the Segal condition
	on $\sd(\nerve\dsD)$ follows from the decomposition-space
	condition on $\nerve\dsD$.
	
	If $\ds{E} \to \dsD$ is CULF, its naturality square along any active map
	is cartesian.  The bottom face maps of $\tw(\ds{E})$ are given by
	(composites of) inner---hence active---maps in $\ds{E}$, and similarly for $\dsD$, so the
	naturality square along any bottom face of $\tw(\ds{E})\to\tw(\dsD)$ is
	again cartesian, as required for it to be a right fibration.
\end{proof}

\begin{remark}
  The object part of \cref{lemma.tw_rFib} has been observed also by
  Bergner, Osorno, Ozornova, Rovelli, and
  Scheimbauer~\cite{Bergner-et.al:1807.05069}, who furthermore
  establish the following converse result (in the more general setting
  of simplicial objects $X$ in a combinatorial model category):
  {\em $\sd(X)$ is Segal if and only if $X$ is
  a decomposition space}. This interpretation of their result 
  depends on the recent result from~\cite{FGKUW} that {\em every $2$-Segal space is 
  unital}.
\end{remark}

\chapter{From category of elements to twisted arrow category}

In this section we describe the natural transformation from 
categories of elements to twisted arrow categories. First we shall need a few basic facts about $\QQQ$ and the ``last-vertex map.''

\begin{lemma}\label{lemma.Qtoppres=act}
  The functor $Q$ sends top-preserving maps to active maps.
\end{lemma}
\begin{proof}
  If $f\colon  [m] \to [n]$ preserves the top element (that is, $f(m)=n$),
  then $Q(f)$ is the map $f\op\join f \colon  [m]\op\join [m] \to 
  [n]\op\join[n]$, which clearly preserves both the bottom element
  $m'$ and the top element $m$.
\end{proof}

The ``last-vertex map'' of \cite{Waldhausen} is a natural
transformation $\last\colon\nerve \circ \el \Rightarrow \id_{\sset}$, which sends 
a simplex $\sigma$ in $(\nerve \el X)_0$ to its last vertex.  Its
value in higher simplicial degree is given by the following lemma.
\begin{lemma}\label{lemma.last_vertex}
For any $k$, let $\varphi \in (\nerve\sxcat)_k$ denote a sequence of maps $[n_0]\stackrel{f_1}\to[n_1]\stackrel{f_2}\to\cdots\stackrel{f_k}\to[n_k]$ in $\sxcat$.  Then there is a unique commutative diagram $B(\varphi)$ of the form
\[\begin{tikzcd}[column sep=large]
	{[0]}\ar[r, "d^\top"]\ar[d, "\beta(n_0)"']&
	{[1]}\ar[r, "d^\top"]\ar[d, "\beta(f_1)"']&
	\cdots\ar[r, "d^\top"]&
	{[k]}\ar[d, "\beta(f_k)"']\\
	{[n_0]}\ar[r, "f_1"']&
	{[n_1]}\ar[r, "f_2"']&
	\cdots\ar[r, "f_k"']&
	{[n_k]}
\end{tikzcd}\]
for which all the vertical maps are top preserving,
and all the maps in the top row are $d^\top$.
\end{lemma}
The proof is straightforward.  For a hint, see the proof provided for
the next lemma, where we give a two-sided refinement of this construction.

\begin{lemma}
	\label{lemma.fund_alt}
For any $k$, let $\varphi \in (\nerve\sxcat)_k$ denote a sequence of maps $[n_0]\stackrel{f_1}\to[n_1]\stackrel{f_2}\to\cdots\stackrel{f_k}\to[n_k]$ in $\sxcat$.  Then there is a unique commutative diagram $A(\varphi)$ of the form
\begin{equation}\label{eqn.fund_lemma}
\begin{tikzcd}[column sep=large]
	\QQQ[0]\ar[r, "\QQQ\left(d^\top\right)"]\ar[d, ->|, "\alpha(n_0)"']&
	\QQQ[1]\ar[r, "\QQQ\left(d^\top\right)"]\ar[d, ->|, "\alpha(f_1)"']&
	\cdots\ar[r, "\QQQ\left(d^\top\right)"]&
	\QQQ[k]\ar[d, ->|, "\alpha(f_k)"']\\
	{[n_0]}\ar[r, "f_1"']&
	{[n_1]}\ar[r, "f_2"']&
	\cdots\ar[r, "f_k"']&
	{[n_k]}
\end{tikzcd}
\end{equation}
i.e.\ for which all the vertical maps are active and all the top maps are of the form $\QQQ(d^\top)$.
\end{lemma}
\begin{proof}
When $k=0$ it is clear: there is a unique active map
$\QQQ[0]=[1]\activeto[n_0]$; call it $\alpha(n_0)$. The result now
follows from the obvious fact that for any map $h\colon \QQQ[m]\to[n]$ in
$\sxcat$, the following solid arrow diagram admits a unique active
extension as shown:
\[
\begin{tikzcd}[column sep=large]
	\QQQ[m]\ar[r, "\QQQ(d^\top)"]\ar[dr, "h"']&\QQQ[m{+}1]\ar[d, dotted, ->|]\\
	&{[n]}
\end{tikzcd}
\qedhere
\]
\end{proof}

\begin{remark}\label{rem.long_edge}
  In the $k=1$ case of \cref{lemma.fund_alt}, 
\[
\begin{tikzcd}[column sep=large]
	{Q[0]}
    \ar[d, ->|, "\alpha\left(m\right)"']
	\ar[r, "\QQQ\left(d^\top\right)"]    &
	{Q[1]}
    \ar[d, ->|, "\alpha\left(f\right)"]
    \\
	{[m]}\ar[r, "f"'] &
	{[n]}    ,
\end{tikzcd}
\]
the map $\alpha(f) \colon [3] \activeto [n]$
is described explicitly for any $f\colon[m]\to[n]$ as 
follows:
\[
0\mapsto 0, \qquad 1 \mapsto f(0), \qquad 2\mapsto f(m), \qquad 3 \mapsto n .
\]
In particular, if $f$ itself is active, so that $f(0)=0$ 
and $f(m)=n$, then 
the map $\alpha(f)$ is 
doubly degenerate, in the sense that it factors through
 $Q(s^0)$ as in the following diagram:
\[
\begin{tikzcd}[column sep=large]
	{[1]}\ar[d, ->|, "\alpha\left(m\right)"']&
	{[3]}\ar[l, ->|, "\QQQ\left(s^0\right)"']\ar[d, ->|, "\alpha\left(f\right)"]\\
	{[m]}\ar[r, ->|, "f"']&
	{[n]}    .
\end{tikzcd}
\]
\end{remark}


\begin{lemma}\label{lemma.Nel_tw}
  There is a natural transformation
  $\Lambda\colon\nerve\circ\el\Rightarrow\sd$ of functors
  $\sset\to\sset$. The degree-$0$ component $(\nerve \el X)_0 \to 
  (\sd X)_0$ sends any simplex in $X$ to its long edge.
\end{lemma}
\begin{proof}
  A $k$-simplex in $\nerve\el(X)$ is a sequence
  $[n_0]\to\cdots\to[n_k]$ together with an $n_k$-simplex $\sigma\in
  X_{n_k}$.  By \cref{lemma.fund_alt}, there is an induced map
  $\QQQ[k]\activeto[n_k]$, and hence $\Delta[k]\to\sd(X)$.  Naturality
  follows from the uniqueness of the diagram in
  \eqref{eqn.fund_lemma}.
\end{proof}

\begin{remark}
The two arguments in lemmas~\ref{lemma.last_vertex} 
and \ref{lemma.fund_alt} can be compared by means of the
last-segment inclusion $L_m \colon [m] \to \QQQ[m]$, by which 
$\cod$ was defined. One finds
that $\Lambda$ mediates between the natural transformations $\last$ and $\cod$
as follows:
\[\begin{tikzcd}[column sep=small]
\nerve \circ \el \ar[rr, Rightarrow, "\Lambda"] \ar[rd, Rightarrow, 
"\last"'] && \sd \ar[ld, Rightarrow, "\cod"] \\
& \id  .&
\end{tikzcd}\]
\end{remark}

\begin{lemma}\label{lemma.Waldhausen}
For any sequence $\varphi \in (\nerve \sxcat)_k$ as in 
lemmas~\ref{lemma.last_vertex} 
and \ref{lemma.fund_alt}, we have
\[
Q(B(\varphi)) = A(Q(\varphi)) .
\]
\end{lemma}

\begin{proof}
  Just observe that $Q(B(\varphi))$ is a diagram with bottom row
  $Q(\varphi)$, and it satisfies the condition of
  \cref{lemma.fund_alt}, as a consequence of \cref{lemma.last_vertex}
  and \cref{lemma.Qtoppres=act}.  Therefore, by uniqueness in
  \cref{lemma.fund_alt}, the diagram must be $A(Q(\varphi))$.
\end{proof}

	For $X$ a simplicial set with corresponding right fibration
	$p\colon\el (X) \to\sxcat$, there is a functor 
	$\omega_X\colon\el (\sd X)\Rightarrow \el (X)$ induced by pullback along $\QQQ$:
\begin{equation}\label{eqn.name_omega}
\begin{tikzcd}
	\el(\sd X)\ar[r, "\omega_X"]\ar[d, "\sd p"'] \drpullback &
	\el(X)\ar[d, "p"]\\
	\sxcat\ar[r, "\QQQ"']&
	\sxcat   
\end{tikzcd}
\end{equation}
It sends an $n$-simplex in $\sd (X)$ to the 
corresponding $(2n{+}1)$-simplex in $X$.  These functors assemble 
into a natural transformation
\[
\omega  \colon\el\circ\sd\Rightarrow \el .
\]

\begin{lemma}
	\label{lemma.last_nerve_Lambda}
The following diagram in the category of endofunctors on $\sset$ commutes:
\[
\begin{tikzcd}[column sep=small]
  &  \nerve\circ \el\circ\sd \ar[dr, Rightarrow, "\last_{\sd}"] 
  \ar[ld, Rightarrow, "\nerve\omega"'] & \\
	\nerve\circ\el\ar[rr, Rightarrow, "\Lambda"']&& 
	\sd  .
\end{tikzcd}
\]

\end{lemma}
\begin{proof}
  Given an $k$-simplex $x\in \nerve\el(\sd X)_k$, i.e.\ a sequence
  $[n_0]\to\cdots\to[n_k]$ and morphism $\Delta[n_k]\to \sd (X)$,
  applying $\last_{\sd}$ returns the dotted arrow as shown
  (see \cref{lemma.last_vertex}):
  \[\begin{tikzcd}[column sep=large]
	{\Delta[0]}\ar[r, "{\Delta[d^\top]}"]\ar[d]&
	{\Delta[1]}\ar[r, "{\Delta[d^\top]}"]\ar[d]&
	\cdots\ar[r, "{\Delta[d^\top]}"]&
	{\Delta[k]}\ar[d]\ar[dr, dotted, "\last_{\sd X}"]\\
	{\Delta[n_0]}\ar[r]&
	{\Delta[n_1]}\ar[r]&
	\cdots\ar[r]&
	{\Delta[n_k]}\ar[r]&
	Q^* X   .
\end{tikzcd}\]
Here we have written $Q^*X$ instead of the usual $\sd (X)$ 
because we will make use of its
adjoint $Q_!$,
which restricts to $Q$ on representables, i.e.\ $Q_!\Delta[n]=\Delta[2n{+}1]$.
Applying instead
$\Lambda\circ N\omega$ returns the dotted arrow as shown:
  \[\begin{tikzcd}[column sep=large]
	{Q_!\Delta[0]}\ar[r, "Q_!{\Delta[d^\top]}"]\ar[d, ->|]&
	{Q_!\Delta[1]}\ar[r, "Q_!{\Delta[d^\top]}"]\ar[d, ->|]&
	\cdots\ar[r, "Q_!{\Delta[d^\top]}"]&
	{Q_!\Delta[k]}\ar[d, ->|]\ar[dr, dotted, "\Lambda\circ N\omega(X)"]\\
	{Q_!\Delta[n_0]}\ar[r]&
	{Q_!\Delta[n_1]}\ar[r]&
	\cdots\ar[r]&
	{Q_!\Delta[n_k]}\ar[r]&
	X   .
\end{tikzcd}\]
To see that these two dotted maps represent the same $k$-simplex of 
$\sd (X)$ via the $Q_!\dashv Q^*$ adjunction, we invoke \cref{lemma.Waldhausen} and the uniqueness from \cref{lemma.fund_alt}.
\end{proof}

\begin{proposition}\label{prop.cart_nt}
On decomposition spaces, the map $\Lambda$ restricts to a cartesian natural 
transformation $\lambda\colon\el\circ\nerve\Rightarrow\tw$:
\[
\begin{tikzcd}[sep=large]
	\decomp\ar[r, bend left=15pt, start anchor={[xshift=-4pt, 
	yshift=-2pt]}, "\el\circ\nerve"]\ar[r, bend right=15pt, start anchor={[xshift=-4pt, yshift=2pt]}, "\tw"']\ar[r, phantom, "\scriptstyle\Downarrow\lambda"]\ar[d, "\nerve"']&
	\smcat\ar[d, "\nerve"]\\
    \sset\ar[r, bend left=15pt, "\nerve\circ\el"]\ar[r, bend right=15pt, "\sd"']\ar[r, phantom, "\scriptstyle\Downarrow\Lambda"]&\sset
\end{tikzcd}
\]
\end{proposition}
\begin{proof}

If $\dsD$ is a decomposition space, $\tw(\dsD)$ is a
category by \cref{lemma.tw_rFib}, and so is $\el(\nerve \dsD)$.
Because $\nerve\colon\smcat\to\sset$ is fully faithful, the natural
transformation $\Lambda$ lifts uniquely to a 
natural transformation $\lambda$ as shown.

It remains to show that for any CULF map $F\colon \ds{E}\to \dsD$, the diagram 
\begin{equation}\label{eqn.el_tw_nt}
\begin{tikzcd}
	\el(\nerve\ds{E})\ar[r, "\lambda_{\ds{E}}"]\ar[d, "\el(\nerve F)"']&
	\tw(\ds{E})\ar[d, "\tw(F)"]\\
	\el(\nerve \dsD)\ar[r, "\lambda_\dsD"']&
	\tw(\dsD)
\end{tikzcd}
\end{equation}
is cartesian.
On objects and morphisms respectively, this amounts to 
showing that a unique lift exists for any solid-arrow squares 
(arbitrary $\tau,\sigma,f$) as follows:
\[
\begin{tikzcd}
	\Delta[1]\ar[r, "\tau"]\ar[d, ->|, "\alpha(n)"']&
	\ds{E}\ar[d, "F"]\\
	\Delta[n]\ar[r, "\sigma"']\ar[ur, dotted]&
	\dsD
\end{tikzcd}
\hspace{1in}
\begin{tikzcd}
	\Delta[1]\ar[r, "{\QQQ_!\Delta[d^\top]}"]\ar[d, ->|, "\alpha(n_0)"']&[5pt]
	\Delta[3]\ar[r, "\tau"]\ar[d, ->|, "\alpha(f)"']&
	\ds{E}\ar[d, "F"]\\
	\Delta[n_0]\ar[r, "f"']&\Delta[n_1]\ar[r, "\sigma"']\ar[ur, dotted]&
	\dsD
\end{tikzcd}
\]
Here the $\alpha$'s denote
the unique active maps, as in \cref{lemma.fund_alt}. 
These two lifts do indeed exist uniquely because $F$ is CULF.
\end{proof}

\begin{remark}\label{rem.act-to-id}
  Let us spell out in explicit terms the functor $\lambda_\dsD \colon
\el(\nerve \dsD) \to \tw(\dsD)$ for $\dsD$ a decomposition space.
An object in $\el(\nerve \dsD)$ is a pair $([n],\sigma)$ with $[n]\in \sxcat$ and 
$\sigma\colon \Delta[n]\to \dsD$.  Now
$\lambda_\dsD( [n],\sigma )$ is
the long edge of $\sigma$, considered as an object in $\tw(\dsD)$.
An arrow in $\el(\nerve \dsD)$ from $([n'],\sigma')$ to $([n],\sigma)$ is
an arrow $f\colon[n'] \to [n]$ such that $\sigma' = f^* \sigma$. Such an arrow
is sent by $\lambda_\dsD$ to the $3$-simplex 
$\lambda_\dsD(f)=\sigma\circ\alpha(f)$ in the diagram
\[
\begin{tikzcd}[sep=large]
	&\Delta[1] \ar[d,  ->|, "{\QQQ_!\Delta[d^0]}"] &
	\\
	\Delta[1]
	\ar[r, "{\QQQ_!\Delta[d^1]}"]
	\ar[d, ->|, "{\alpha(n')}"']
	&[5pt]
	\Delta[3]
	\ar[rd, dotted, "{\lambda_\dsD(f)}"]
	\ar[d, ->|, "\alpha(f)"']&
	\\
	\Delta[n']\ar[r, "{\Delta[f]}"']&\Delta[n]\ar[r, "\sigma"']&
	\dsD
	\,,
\end{tikzcd}
\]
where the square is as in \cref{lemma.fund_alt} (and more specifically 
\cref{rem.long_edge}). This $\lambda_\dsD(f)$ is
regarded as an arrow in the category $\tw(\dsD)$ with domain 
$\Delta[1] \To{\QQQ_!\Delta[d^1]} \Delta[3] \To{\lambda_D(f)} \dsD$ 
and codomain
$\Delta[1] \To{\QQQ_!\Delta[d^0]} \Delta[3] \To{\lambda_D(f)} \dsD$, 
in the usual way.

In this situation, if $f\colon [n'] \activeto [n]$ is active, then $\lambda_\dsD(f)$ 
is an identity arrow in $\tw(\dsD)$, that is, a $3$-simplex in $\dsD$ of 
the
form ``$\QQQ_!\Delta[s^0]$ of a $1$-simplex'', as noted in
\cref{rem.long_edge}.
\end{remark}

\chapter{Proof of main theorem}\label{sec.proof_main}

\begin{lemma}\label{lemma.easy_rhombus}
  For every decomposition space $\dsD$, the following diagram commutes (up to isomorphism):
\begin{equation}\label{eqn.sliced}
	\begin{tikzcd}[column sep=large]
	  \decomp_{/\dsD}  \ar[r,"\nerve_\dsD"]\ar[rd, "\tw_\dsD"']& 
	  \sset_{/\nerve \dsD} 
	  \ar[r,"\el_{\nerve \dsD}"]\ar[rd, "\sd_{\nerve \dsD}"']& 
	  \RFib_{/\el(\nerve \dsD)} \ar[rd, "\QQQ^*_{\el (\nerve \dsD)}"] \\
	  & \RFib_{/\tw(\dsD)} \ar[r, "\nerve_{\tw (\dsD)}"'] & \sset_{/\sd 
	  (\nerve \dsD)} 
	  \ar[r, "\el_{\sd (\nerve \dsD)}"'] & 
     \RFib_{/\el(\sd \nerve \dsD)}
	\end{tikzcd}
\end{equation}
 where $\QQQ^*\colon\RFib\slice{\sxcat}\to\RFib\slice{\sxcat}$ is pullback along $\QQQ$.
\end{lemma}

\begin{proof}
We already have the following (up-to-iso) commutative diagram:
\[
	\begin{tikzcd}[column sep=large, row sep=small]
	  \decomp  \ar[r,"\nerve"]\ar[rd, "\tw"']& \sset
	  \ar[r,"\el"]\ar[rd, "\sd"']&
	  \RFib_{/\sxcat} \ar[rd, "\QQQ^*"] \\
	  & \RFib \ar[r, "\nerve"'] & \sset 
	  \ar[r, "\el"'] & 
	  \RFib_{/\sxcat}  .
	\end{tikzcd}
\]
	Indeed, the left side side is \cref{lemma.tw_rFib}, and the
	right side is just the translation between simplicial sets and right 
	fibrations over $\sxcat$ given in \cref{eqn.name_omega}.  Now slice it over $\dsD$.  (As always 
	when slicing (up-to-iso) commutative diagrams,
	the result involves the 
	identifications already expressed by the diagrams.  In the 
	present case we use $\sd(\nerve \dsD) \cong \nerve \tw (\dsD)$ and
	$\QQQ^* (\el \nerve \dsD) \cong \el (\sd \nerve \dsD) \cong \el 
	(\nerve \tw \dsD)$.)
\end{proof}

A key ingredient in the proof of the main theorem is to see that
$\lambda_\dsD^*$ provides a sort of splitting of the double 
rhombus diagram from the previous lemma:

\begin{lemma}\label{lemma.split_rhombus}
  For any $\dsD\in\decomp$, the following diagram commutes up to
  natural isomorphism.
  \begin{equation}\label{eqn.prism}
	\begin{tikzcd}[column sep=large]
	  \decomp\slice{\dsD}  \ar[r,"\nerve_\dsD"]\ar[rd, "\tw_\dsD"']& 
	  \sset\slice{\nerve \dsD} 
	  \ar[r,"\el_{\nerve \dsD}"]&
	  \RFib_{/\el (\nerve \dsD)} \ar[rd, "\QQQ^*_{\el (\nerve \dsD)}"] \\
	  & \RFib_{/\tw(\dsD)} \ar[r, "\nerve_{\tw (\dsD)}"'] 
      \ar[ru, "\lambda_\dsD^*"'] 
 & \sset\slice{\sd (\nerve \dsD)}
	  \ar[r, "\el_{\sd (\nerve \dsD)}"'] & [5pt]
    \RFib\slice{\el(\sd \nerve \dsD)}   
	\end{tikzcd}
  \end{equation}
\end{lemma}

\begin{proof}
  The commutativity of the left square
  follows immediately from the fact that $\lambda$ is cartesian, 
  cf.~\cref{prop.cart_nt}; see in particular the pullback square in \cref{eqn.el_tw_nt}.
  
For the right square,
  let $p\colon \cat{F} \to \tw(\dsD)$ be a right fibration. We must establish a map
  $\el(\nerve \cat{F})\to\QQQ^*_{\el (\nerve \dsD)}(\lambda_\dsD^*\cat{F})$ and show it is an isomorphism over $\el(\sd \nerve \dsD)$.
    By \cref{eqn.name_omega}, the functor $\QQQ^*_{\el (\nerve \dsD)}$ is pullback along $\omega_{\nerve\dsD}$ in the following diagram of categories:
  	\[
	\begin{tikzcd}
		\cat{F}' \drpullback \ar[r]\ar[d, "p'"']& \cdot 
		\drpullback \ar[d]\ar[r] & \cat{F} \ar[d, "p"]\\
		\el(\sd\nerve \dsD) \drpullback \ar[r, "\omega_{\nerve \dsD}"] \ar[d] & \el(\nerve \dsD) \ar[r, 
		"\lambda_\dsD"] \ar[d] & \tw \dsD \\
		\sxcat \ar[r, "\QQQ"'] & \sxcat
	\end{tikzcd}
	\]	
Thus we have $\cat{F}'\cong\QQQ^*_{\el (\nerve \dsD)}(\lambda_\dsD^*\cat{F})$.
The nerve of the middle composite,
\[
\nerve\el(\sd\nerve \dsD)\To{\nerve\omega_{\nerve 
\dsD}}\nerve\el(\nerve \dsD)\To{\Lambda\nerve}\sd(\nerve \dsD) ,
\]
is identified with
$\last_{\sd(\nerve \dsD)}$
by \cref{lemma.last_nerve_Lambda,prop.cart_nt}, 
which in particular gives
$\nerve\lambda\cong\Lambda\nerve$.
The naturality square for the last-vertex map
\[
\begin{tikzcd}[column sep=large]
	\nerve\el(\nerve \cat{F})\ar[r, "\last_{\nerve \cat{F}}"]\ar[d, 
	"\nerve\el(\nerve p)"']&
	\nerve \cat{F}\ar[d, "\nerve p"]\\
	\nerve\el(\sd\nerve \dsD)\ar[r, "\last_{\sd(\nerve \dsD)}"']&
	\sd(\nerve \dsD)
\end{tikzcd}
\]
induces a morphism $\nerve\el(\nerve \cat{F})\to \nerve \cat{F}'$ by the
universality of $\cat{F}'$ as a pullback, and the fact that $\nerve\colon\smcat\to\sset$ preserves
limits.  Since $\nerve$ is fully faithful, we obtain our desired
comparison map $u\colon\el(\nerve \cat{F})\to \cat{F}'$ of discrete
fibrations over $\el(\sd\nerve \dsD)$.  It is enough to check that the
restriction of $u$ to each fiber is a bijection; to do so we describe
these fibers and the map $u$ between them in concrete terms.

An object in $\el(\sd \nerve \dsD)$ is a pair $([n],\sigma)$ where
$\sigma\colon\Delta[n]\to\sd(\nerve \dsD)$.  As an $n$-simplex in $\tw(\dsD)$,
it is sent by $\omega_{\nerve \dsD}$ to the same pair $([n],\sigma)$,
but where now $\sigma$ is considered a $(2n{+}1)$-simplex in $\dsD$.
Applying $\lambda_\dsD$ returns the long edge of the simplex, $\Delta[1]
\activeto \Delta[2n{+}1] \to \dsD$, which we denote $\ell \in \ob
\tw(\dsD)$.  The fiber of $\cat{F}'$ over $\sigma$ is the $p$-fiber
over $\ell$, that is the discrete set $\{z\in \ob \cat{F} \mid
p(z)=\ell\}$.  In other words, we can identify an object in this fiber
with a commutative diagram
\begin{equation}\label{eqn.fibration_last}
\begin{tikzcd}
	{\Delta[0]}\ar[r, "z"]\ar[d, "n"']&
	\cat{F}\ar[d, "p"]\\
	{\Delta[n]}\ar[r, "\sigma"']&
	\sd(\nerve \dsD)  .
\end{tikzcd}
\end{equation}
  On the other hand, an object in $\el(\nerve \cat{F})$ over $\sigma$
  can be identified with an $n$-chain of arrows $z_0\to \cdots \to
  z_n$ in $\cat{F}$, lying over $\sigma$.  The comparison functor
  $u\colon \el(\nerve \cat{F}) \to \cat{F}'$ sends such an $n$-chain
  to its last element, $z_n$.  Thus it suffices to show that there is
  a unique lift in \cref{eqn.fibration_last}.  But this is exactly the
  condition that $p$ is a right fibration.
\end{proof}

\begin{corollary}\label{cor.main}
With notation as in \cref{eqn.sliced,eqn.prism}, we have natural isomorphisms
\[
	\partial_{\sd(\nerve\dsD)}\circ Q^*_{\el(\nerve\dsD)}\cong
	\sd_{\nerve\dsD}\circ\partial_{\nerve\dsD}
	\qquad\text{and}\qquad
	\nerve_{\tw (\dsD)}\cong
	\partial_{\sd (\nerve \dsD)}\circ \QQQ^*_{\el 
	(\nerve \dsD)}\circ\lambda_\dsD^*.
\]
\end{corollary}
\begin{proof}
Using that $\partial$ and $\el$ are inverse equivalences of categories between presheaves and right fibrations, these statements follow from 
\cref{eqn.sliced,eqn.prism}, right parts.
\end{proof}

\begin{lemma}\label{lemma.def_untw}
  For any decomposition space $\dsD$, 
  there exists a functor
  $\untw{\dsD}\colon\RFib\slice{\tw(\dsD)}\to\decomp\slice{\dsD}$ 
  with
  \[
  \untw{\dsD} \circ \tw_\dsD \cong \id_{\decomp\slice{\dsD}} .
  \]
  Furthermore, both the left-to-right and the right-to-left squares commute up to natural isomorphism:
\begin{equation}\label{eqn.forward_backward}
\begin{tikzcd}[sep=large]
	\decomp\slice{\dsD}\ar[r, bend left=10pt, "\tw_\dsD"]\ar[d, "\nerve_\dsD"']&
	\RFib\slice{\tw(\dsD)}\ar[d, "\lambda_\dsD^*"]\ar[l, bend left=10pt, dotted, "\untw{\dsD}"]\\\
	\sset\slice{\nerve \dsD}\ar[r, bend left=10pt, "\el_{\nerve \dsD}"]&
	\RFib\slice{\el(\nerve \dsD)}   . 
	\ar[l, bend left=10pt, "\partial_{\nerve \dsD}"]
\end{tikzcd}
\end{equation}
\end{lemma}

\begin{proof}
  We already know that the left-to-right square commutes by \cref{lemma.split_rhombus} (left part).
  To define $\untw{\dsD}$ making the right-to-left square commute, it
  suffices to show that for any right fibration $p\colon
  \cat{F}\to\tw(\dsD)$, the simplicial map $\partial_{\nerve
  \dsD}\lambda_\dsD^*(p)$ is CULF, because then it lands in $\decomp$ 
  (see the preliminaries on decomposition spaces, p.\pageref{page.CULF_over_decomp}).
  CULFness amounts to certain squares being pullbacks. To 
  express this cleanly, 
  consider 
  the right fibration $q\colon U \to V$ 
  defined by the following pullbacks:
%
  	\[
	\begin{tikzcd}
		U \drpullback \ar[r]\ar[d, "q"']& \cdot 
		\drpullback \ar[d]\ar[r] & \cat{F} \ar[d, "p"]\\
		V \drpullback \ar[r] \ar[d] & \el(\nerve \dsD) \ar[r, 
		"\lambda_\dsD"'] \ar[d] & \tw(\dsD) \\
		\sxcat_{\operatorname{active}} \ar[r, "\text{incl}"'] & 
		\sxcat 
	\end{tikzcd}
	\]
  The statement is that the right fibration $q\colon U \to V$ 
corresponds to a {\em cartesian} natural transformation 
  \[
\begin{tikzcd}[column sep=large]
	\sxcat_{\operatorname{active}}\op\ar[r, bend left=20pt, "\partial U"]\ar[r, bend 
	right=20pt, 
	"\partial V"']\ar[r, phantom, "\scriptstyle\Downarrow"]&
	\smset ,
\end{tikzcd}
\]
  and this condition in turn can be read off directly on $q$: we need to check
  that for every active map $f \colon [m] \activeto [n]$
  the following square is a pullback of sets:
  	\[
	\begin{tikzcd}
       U_{m}  \ar[d, "q_{m}"']& U_n  \ar[l, "f^*"'] 
	   \ar[d, "q_{n}"] \\
	   V_{m} & V_n    .   
	   \ar[l, "f^*"] 
   \end{tikzcd}
   \]
  We compare the $q_n$-fiber over a point
  $\Delta[n]\stackrel\sigma\to \dsD$ in $V_n$ with the $q_m$-fiber 
  over
  the corresponding point $\Delta[m] \stackrel{\Delta[f]}\to \Delta[n]
  \stackrel\sigma\to \dsD$ in $V_m$. 
  We do this by computing
  these $q$-fibers in terms of $p$-fibers, using the explicit 
  description of $\lambda_\dsD$ from \cref{rem.act-to-id}. But we 
  already noted in \cref{rem.act-to-id} that
  $\lambda_\dsD$ sends any arrow in $\el(\nerve\dsD)$ lying over an 
  active map to an identity, so the $q_n$ and $q_m$ 
  fibers coincide for every $\sigma\colon \Delta[n] \to \dsD$ and 
  every active map $f\colon [m]\activeto [n]$, and therefore the square is
  a pullback.
     
  Finally the main statement follows easily: 
  we first read off from the commutativity of \cref{eqn.forward_backward} that
\[
\nerve_{\dsD} \circ \untw{\dsD} \circ \tw_\dsD \cong
\partial_{\nerve \dsD} \circ \lambda_\dsD^* \circ \tw_\dsD \cong \partial_{\nerve \dsD} \circ 
\el_{\nerve \dsD} \circ \nerve_{\dsD}
\cong \nerve_{\dsD}.
\]
Since the nerve functor is fully faithful, so is the slice 
$\nerve_{\dsD}$, and we have established $\untw{\dsD} \circ \tw_\dsD \cong \id_{\decomp\slice{\dsD}}$.
\end{proof}

\begin{theorem}[Main theorem]\label{thm.main}
For any decomposition space $\dsD$, we have natural inverse equivalences of categories
\[
\begin{tikzcd}[column sep=large]
	\decomp\slice{\dsD}\ar[r, shift left, "\tw_\dsD"]&
	\RFib\slice{\tw(\dsD)}.\ar[l, shift left, "\untw{\dsD}"]
\end{tikzcd}
\]
\end{theorem}
\begin{proof}
We established $\untw{\dsD} \circ \tw_\dsD \cong \id_{\decomp\slice{\dsD}}$ in
\cref{lemma.def_untw}.  For the other direction, since nerve is fully
faithful, it suffices to prove that
$\nerve_{\tw(\dsD)}\circ\tw_\dsD\circ\untw{\dsD}\cong\nerve_{\tw(\dsD)}$;
this is the outer square
in the diagram below:
  
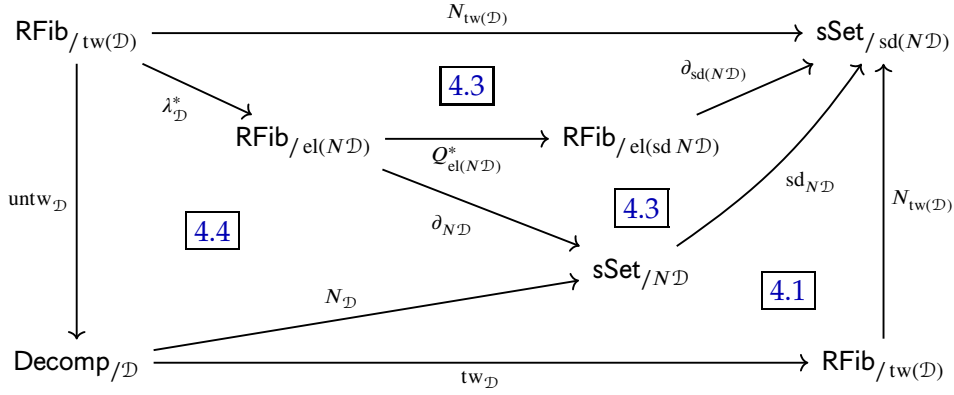
\[
  \begin{tikzcd}
	\RFib\slice{\tw(\dsD)}
		\ar[rrrr, "\nerve_{\tw(\dsD)}"]
		\ar[dddd, "\untw{\dsD}"']
		\ar[dr, "\lambda_\dsD^*"']&&
	{}&&
	\sset\slice{\sd (\nerve \dsD)}
	\\
	&\RFib\slice{\el (\nerve \dsD)}
		\ar[dd, phantom, xshift=-7ex, yshift=-1ex, "\fbox{\ref{lemma.def_untw}}"]
		\ar[rr, "Q^*_{\el (\nerve \dsD)}"']
		\ar[ddrr, "\partial_{\nerve \dsD}"']&
	{}
		\ar[u, phantom, yshift=0.6ex, "\fbox{\ref{cor.main}}"]&
	\RFib\slice{\el (\sd \nerve \dsD)}
		\ar[ur, "\partial_{\sd (\nerve \dsD)}"]
		\ar[dd, phantom, "\fbox{\ref{cor.main}}"]
	\\
	&&&&&
	\\[-10pt]
	&{}&&\sset\slice{\nerve \dsD}
		\ar[uuur, bend right=10pt, "\sd_{\nerve \dsD}"']
		\ar[dr, phantom, xshift=2ex, yshift=2ex, "\fbox{\ref{lemma.easy_rhombus}}"]
	\\[-5pt]
	\decomp\slice{\dsD}
		\ar[urrr, "\nerve_\dsD"]
		\ar[rrrr, "\tw_\dsD"']&&&&
	\RFib\slice{\tw(\dsD)}
		\ar[uuuu, "\nerve_{\tw(\dsD)}"']
  \end{tikzcd}
  \]
The up-to-iso commutativity of each of the squares inside was proven earlier as indicated.
\end{proof}

\begin{remark}
  All the proof ingredients are readily seen to be natural in $\dsD$.  
  In fact the main theorem can be seen as the $\dsD$-component of a 
  natural equivalence of $\smcat$-valued functors
  \[
  \begin{tikzcd}[column sep=large]
	  \decomp\ar[r, bend left=20pt, start anchor={[xshift=-4pt, yshift=-3pt]}, "\decomp\slice{-}"]\ar[r, bend right=20pt, start anchor={[xshift=-4pt, yshift=3pt]},  "\RFib\slice{\tw(-)}"']&
	  \smcat\ar[l, phantom, pos=.6, "\simeq"]   .
  \end{tikzcd}
  \]
\end{remark}

\begin{corollary}
  For $\dsD$ a decomposition space and  $\cat{F}\to \tw(\dsD)$ a right fibration,
  $\cat{F}$ is again the twisted arrow category of a decomposition space.
\end{corollary}

\begin{corollary}
The category $\decomp$ is locally cartesian closed.
\end{corollary}



\Addresses
\vspace{-.2in}

\end{document}